\documentclass[12pt,leqno,twoside,elsarticle-num-names]{elsarticle}
\usepackage{amssymb}
\usepackage{latexsym}
\usepackage{graphicx}
\usepackage{epsfig}
\usepackage{srcltx}
\usepackage{subcaption} 

\newtheorem{theorem}{Theorem}[section]
\newtheorem{corollary}[theorem]{Corollary}
\newtheorem{lemma}[theorem]{Lemma}
\newtheorem{proposition}[theorem]{Proposition}

\newtheorem{remark}[theorem]{\sc Remark}

\newproof{proof}{Proof}
\newproof{proofp}{Proof of Proposition \ref{p:main}}

\renewcommand{\int}{{\mathrm{int}}}

\newcommand{\tF}{F^{\pitchfork}}

\newcommand{\bC}{{\mathbb C}}

\newcommand{\bZ}{{\mathbb Z}}

\newcommand{\bA}{{\mathbb A}}

\begin{document}

\begin{frontmatter}

\title
 {Non reduced plane curve singularities with $b_1(F)=0$  and Bobadilla's question}

\author{\sc Dirk Siersma}  

\address{Institute of Mathematics, Utrecht University, 
PO Box 80010, \ 3508 TA Utrecht
The Netherlands.}

\ead{D.Siersma@uu.nl}





\begin{abstract}
If the first Betti number of the Milnor fibre of a plane curve singularity is zero,
then the defining function is equivalent to $x^r$.

\end{abstract}

\begin{keyword}
\texttt{ Milnor fibre \sep equisingular \sep 1-dimensional critical locus \sep
Bobadilla's conjecture \sep
\MSC[2010] 14H20 \sep 32S05 \sep 32S15}
\end{keyword}

\date{\today}


\end{frontmatter}


\setcounter{section}{0}

\section{Introduction}\label{s:intro}

Let $f: \bC^n \to \bC$ be a holomorphic function germ. What can be said about functions whose Milnor fibre $F$ has the property $b_i(F) = 0$ for all $i \ge 1$ ? If $F$ is connected then $f$   is non-singular and equivalent to a linear function by A'Campo's trace formula. The remaining question: \emph{What happens if $F$ is non-connected~?} is only relevant for non-reduced plane curve singularities.

This question is related to a recent paper \cite{HM}. That paper contains a statement about the so-called Bobadilla conjectures \cite{Bo} in case of plane curves. The invariant $\beta= 0$, used by Massey \cite{Ma}  should imply that the singular set of $f$ is a smooth line.

In this note we give a short topological proof of a stronger statement.

\begin{proposition} \label{p:main}
If the first Betti number of the Milnor fibre of a plane curve singularity is zero,
then the defining function is equivalent to $x^r$. 
\end{proposition}

\begin{corollary} \label{c:line}
In the above case the singular set is a smooth line and the system of transversal singularities is trivial. 
\end{corollary}

\section{Non-reduced plane curves}

Non-isolated plane curve singularities have been thoroughly studied by Rob Schrauwen in his dissertation \cite{Sch1}.  Main parts of it are published as \cite{Sch2} and \cite{Sch3}. The above Proposition \ref{p:main} is an easy consequence of his work.
\medskip

We can assume that $f = f_1^{m_1}. \cdots . f_r^{m_r}$
(partition in powers of reduced irreducible components).

\begin{lemma}
Let $d = \gcd (m_1,\cdots,m_r)$
\begin{itemize}
\item[(a.)] $F$ has $d$ components, each diffeomorphic to the Milnor fibre $G$ of
$g = g_1^{\frac{m_1}{d}}. \cdots . g_r^{\frac{m_r}{d}}$. The Milnor monodromy of $f$ permutes these components,
\item[(b.)] if $d=1$ then $F$ is connected.
\end{itemize}
\end{lemma}

\begin{proof}
(a.) Since $f = g^d$ the fibre $F$ consists of $d$ copies of $G$.\\
(b.) We recall here the reasoning from \cite{Sch1}. Deform the reduced factors $f_i$ into $\hat{f}_i$ such that 
the product  $ \hat{f}_1. \cdots .\hat{f}_r = 0 $ contains the maximal number of  double points (cf. Figure \ref{f:Qpoints}). This is called a network deformation by Schrauwen. The corresponding deformation $\hat{f}$ of $f$  near such a point  has local equation are of the form $x^p y^q = 0$ (point of type $D[p,q]$).

\begin{figure}[htbp]
    \centering
        \includegraphics[width=6cm]{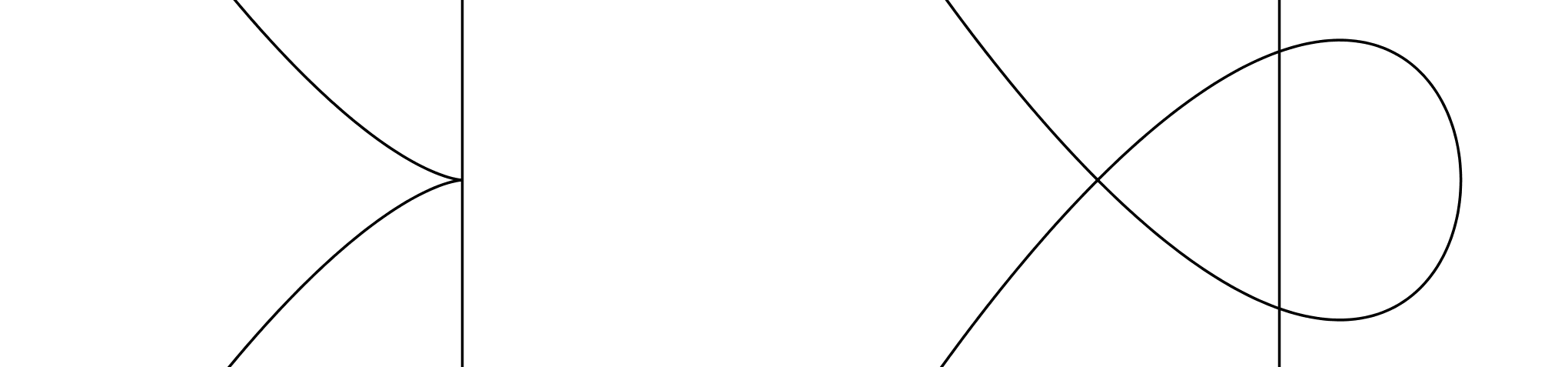}
    \caption{Deformation to maximal number of double points.}
    \label{f:Qpoints}
\end{figure}

Near  every branch $\hat{f}_i=0$ the Milnor fibre is a $m_i$-sheeted covering of the zero-locus, except in the $D[p,q]$-points.  We construct the Milnor fibre $F$ of $f$ starting with $S=\sum m_i$ copies of the affine line $\bA$. 
Cover the {\it i}th branch with $m_i$ copies of $\bA$ and delete $(p+q)$ small discs around the $D[p,q]$-points. Glue in the holes  $\gcd (p,q)$ small annuli (the Milnor fibres of $D[p,q]$). The resulting space is the Milnor fibre $F$ of $f$.

A hyperplane section of a generic at a generic point of $\hat{f}_i=0$ defines a transversal Milnor fibre $\tF_1$.
Start now the construction of $F$ from $\tF_1$, which consists of $m_1$ cyclic ordered points. 
As soon as $f_1=0$ intersects $f_k=0$ it connects the sheets of $f_1=0$  modulo $m_k$.  Since $\gcd (m_1,\cdots,m_r) = 1$  we connect all sheets.
\end{proof}

\begin{proofp}
If $b_1(F)=0$, then also $b_1(G)=0$. The  Milnor monodromy has trace$(T_g)$ = 1. According to
A'Campo's observation  \cite{AC} $g$ is regular: $g=x$.
It follows that $f = x^r$.
\end{proofp}

\section{Relation to Bobadilla's question}

We consider first in any dimension $f: \bC^{n+1} \to \bC$ with a 1-dimensional singular set, see especially the 1991-paper \cite{Si} for definitions, notations and statements. 

 We focus on the group $H_n (F,F^{\pitchfork})$ 
which occurs in two exact sequences on  p. 468 of \cite{Si}:

\[ 0 \rightarrow H_{n-1}^{\mbox{\tiny f}}(F) \rightarrow H_{n-1}( F^{\pitchfork})
\rightarrow H_n(F)  \oplus H_{n-1}^{\mbox{\tiny t}} (F)
\rightarrow 0 \]

\[ 0 \rightarrow H_{n}(F) \rightarrow  H_n (F,F^{\pitchfork}) 
\rightarrow H_{n-1} (F^{\pitchfork}) 
\rightarrow H_{n-1} (F)
\rightarrow 0 \]

Here $F^{\pitchfork}$ is the disjoint  union of the transversal Milnor fibres $F^{\pitchfork}_i$, one for each irreducible branch of the 1-dimensional singular set.\footnote{ $F^{\pitchfork}$ was originally denoted by $F'$. In the second sequence a misprint $n$ in the third term has been changed to $n-1$.} 

Note that $ H_{n}(F) ,\;  H_n (F,F^{\pitchfork})$  and $H_{n-1}( F^{\pitchfork})$ are free groups.  $H_{n-1}(F)$     can have torsion,
 we denote its free part by $H_{n-1}(F)^{\mbox{\tiny f}}$    and its torsion part by $H_{n-1}(F)^{\mbox{\tiny t}}$. All homologies here are taken over $\bZ$, but also other coefficents are allowed.

From both sequences it follows that the $\beta$-invariant, introduced in \cite{Ma} has a 25 years history, since is nothing else than:
\[ \dim H_n(F,F^{\pitchfork}) = b_n - b_{n-1} + \sum \mu_i^{\pitchfork} : = \beta \]

From this definition is immediately clear that $\beta \ge 0$ and that $\beta$ is topological. The topological defintion has as direct consequence:

\begin{proposition}\label{beta=0}
 Let $f: \bC^{n+1} \to \bC$ with a 1-dimensional singular set, then:
\[ \beta = 0 \; \;  \Leftrightarrow  \; \; \chi(F) = 1 + (-1)^n \sum \mu_i^{\pitchfork} \; \;  \Leftrightarrow  \; \;
H_n(F,\bZ) = 0 \;  \mbox{\rm and} \; H_{n-1}(F,\bZ) = \bZ^{\sum \mu_i} \]

\end{proposition}

The  original Bobadilla conjecture C \cite{Bo} was  in \cite{Ma}  generalized to the reducible case as follows: \emph{Does $\beta = 0$ imply that the singular set is smooth?} As consequence of our main Proposition \ref{p:main} we have:
\begin{corollary}
In the curve case $\beta = 0$ implies that the singular set is smooth; and that the function is equivalent to $x^r$.
\end{corollary}

\begin{remark}
In \cite{HM} the first part of this corollary was obtained with the help of L\^e numbers.
\end{remark}

\begin{remark}
From the definition  $\beta = H_n (F,F^{\pitchfork})$   follow direct and short proofs of several statements from \cite{Ma}. 
\end{remark}

An other consequence from \cite{Si} is the composition of surjections:

\[ H_{n-1}(F^{\pitchfork})= \oplus \bZ^{\mu_i} \twoheadrightarrow H_{n-1}(\partial_2 F) = \oplus \frac{\bZ^{\mu_i}}{A_i-I}  \twoheadrightarrow H_{n-1}(F)  \]

From this follows:

\begin{proposition} \label{beta=bound}
If $\dim H_{n-1}(F) = \sum \mu_i$ (upper bound) then 
\begin{itemize}
\item[a.] $ H_{n-1}(\partial_2 F)$ and $ H_{n-1}(F)$ are free and isomorphic to $ \bZ^{\sum \mu_i}$. 
\item[b.] All transversal monodromies $A_i$ are the identity.
\end{itemize}
\end{proposition}

The second part of \cite{Ma} contains an elegant statement about $\beta=1$  via the A'Campo trace formula.
Also the reduction of the generalized Bobadilla conjecture to the (irreducible) Bobadilla conjecture.
As final remark: The great work (the irreducible case) has still has to be done!
Together with the  L\^e-conjecture this seems to be an important question in the theory of hypersurfaces 1-dimensional singular sets.\\

\section*{References}

\end{document}